\documentclass[11pt]{article}
\usepackage{fullpage}
\usepackage{authblk}
\usepackage{amsmath,amsthm,amsfonts,amssymb,amstext,mathtools}
\usepackage{microtype}
\usepackage{enumitem}
\usepackage[hidelinks]{hyperref}
\newtheorem{theorem}{Theorem}[section]
\newtheorem{lemma}[theorem]{Lemma}
\newtheorem{proposition}[theorem]{Proposition}
\newtheorem{definition}[theorem]{Definition}
\newcommand{\R}{\mathbb R}
\newcommand{\Z}{\mathbb Z}

\makeatletter
\renewcommand\section{%
  \@startsection{section}{1}{\z@}%
    {-3.5ex \@plus -1ex \@minus -.2ex}%
    {2.3ex \@plus .2ex}%
    {\large\bfseries}%
}
\renewcommand\subsection{%
  \@startsection{subsection}{2}{\z@}%
    {-3.25ex \@plus -1ex \@minus -.2ex}%
    {1sp}%
    {\normalsize\bfseries}%
}
\renewcommand\subsubsection{%
  \@startsection{subsubsection}{3}{\z@}%
    {-3.25ex \@plus -1ex \@minus -.2ex}%
    {1sp}%
    {\normalfont\normalsize}%
}
\makeatother

\title{{\Large\bfseries On Kernel Mengerian Orientations of Line Multigraphs}\thanks{Supported in part by the
National Natural Science Foundation of China (No.\,12671377)
and the Shandong Provincial Natural Science Foundation (No.\,ZR2025MS104).}}
\author{Han Xiao\thanks{Corresponding author. Email: hxiao@ouc.edu.cn.}}
\affil{School of Mathematical Sciences\\Ocean University of China\\Qingdao, China}
\date{}
\begin{document}
\numberwithin{equation}{section}
\maketitle

\openup 1.2\jot

\begin{abstract}
A kernel of a digraph is an independent set $K$ such that every vertex
outside $K$ has an arc to a vertex in $K$.
We present a polyhedral description of kernels in clique-acyclic
orientations of line multigraphs.
Given a digraph $D$, let $FK(D)$ denote its fractional kernel polytope,
and let $\sigma(D)$ denote the linear system defining $FK(D)$.
A digraph $D$ is called kernel perfect if every induced subdigraph has a
kernel, called kernel ideal if the fractional kernel polytope of every
induced subdigraph is integral, and called kernel Mengerian if the
defining system of every induced subdigraph is totally dual integral.
We show that a clique-acyclic orientation of a line multigraph is kernel
ideal if and only if it is kernel Mengerian.
These equivalent properties hold precisely when the orientation is
kernel perfect and has no even-cycle obstruction.


\end{abstract}

\newpage
\section{Introduction}
\label{intro}

A \emph{kernel} of a digraph is an independent set to which every
vertex outside the set has an outgoing arc.
A digraph is \emph{kernel perfect} if every induced subdigraph has a kernel.
The fractional kernel polytope $FK(D)$ is defined by nonnegativity,
independence, and domination constraints; its integral points
are precisely the incidence vectors of kernels.
We call $D$ \emph{kernel ideal} if $FK(D')$ is integral for every
induced subdigraph $D'$, and \emph{kernel Mengerian} if the specified
inequality system is totally dual integral on every such restriction.

We study orientations of line multigraphs with arbitrary root degrees
and parallel multiplicities.
An orientation is \emph{clique-acyclic} if the one-way arcs
form an acyclic digraph within each clique.
Borodin \textit{et al.} \cite{BoroKost98} characterized kernel
perfectness in this class by the absence of odd root cycles whose
edges induce a directed cycle.
Our main result, Theorem~\ref{thm:general-characterization}, states that
for clique-acyclic orientations of line multigraphs,
\[
 \mathrm{kernel\ ideal}\quad\Longleftrightarrow\quad
 \mathrm{kernel\ Mengerian}\quad\Longleftrightarrow\quad
 \begin{gathered}
 \text{kernel perfect with no even-cycle obstruction}.
 \end{gathered}
\]
The additional obstruction consists of directed even root cycles and
connecting edges whose alternating-choice constraints are incompatible.
It is defined precisely in Section~\ref{sec:configurations} through
nonbipartiteness of an auxiliary graph.
Proposition~\ref{thm:mixed-obstruction} shows that this condition remains
necessary on bipartite roots.

Kernels of preference-induced line orientations correspond to stable
matchings.
Rothblum \cite{Roth92} proved integrality of the stable-marriage
formulation, and Kir\'{a}ly and Pap \cite{KiraPap08} proved its total dual
integrality.
Chen \textit{et al.} \cite{ChenDing12} treated integrality and total
dual integrality for general simple preference systems.
A clique-acyclic orientation of a line multigraph need not admit a global
preference realization, as the same example shows.
We instead reconstruct endpoint orders where the fractional support
allows it and retain domination rows when zero coordinates are removed.
This gives $1/2$-integrality in induced-minimal failures and extracts
the cycle obstruction.
A separate argument with integer row lattices proves TDI of the
original system.

The polyhedral description of kernels is also discussed in
\cite{EgresKernelPolytope}.
Chen \textit{et al.} \cite{ChenChenZang16} studied a formulation with
pairwise independence constraints in place of clique constraints.
Here the terms kernel ideal and kernel Mengerian refer to the
linear system $\sigma(D)$.
Section~\ref{sec:preliminaries} gives the required polyhedral facts,
Section~\ref{sec:kernel} states the characterization and its obstruction,
and Section~\ref{sec:kernel-proof} proves it.

\section{Preliminaries}
\label{sec:preliminaries}

All graphs and digraphs are finite.
The root $H$ is a loopless multigraph whose parallel edge copies are
distinct, and $\delta_H(u)$ denotes its star at $u$.
For a coordinate vector $x$ and a subset $U$ of its index set, write
$x(U)=\sum_{i\in U}x_i$.
A polytope in $\R^n$ is \emph{$1/2$-integral} if each of its vertices belongs to
$\tfrac12\Z^n$.
We regard the empty polytope as integral and set
$\operatorname{conv}(\varnothing)=\varnothing$.
A rational system $Ax\le b$ is \emph{totally dual integral} (TDI) if
every integral objective with a finite primal optimum admits an
integral optimal dual vector.
A TDI system with integral right-hand side defines an integral
polyhedron \cite{EdmoGile84}.
Nonnegativity is always included among the rows when discussing TDI.

\subsection{Kernels and their polyhedral description}

Let $D=(V,A)$ be a loopless digraph.
Independence and cliques refer to its underlying undirected graph;
write $\mathcal C(D)$ for the family of all nonempty cliques, including
singletons.
A set $K\subseteq V$ \emph{dominates} a vertex $v$ if $v\in K$ or
$v\to u$ for some $u\in K$.
A \emph{kernel} is an independent set that dominates every vertex of $D$.
For $v\in V$ let $N_D^+(v)=\{u:v\to u\in A\}$ and
$N_D^+[v]=\{v\}\cup N_D^+(v)$, counting each target once.
We omit the subscript $D$ when it is unambiguous.
Let $\sigma(D)$ denote the linear system consisting of nonnegativity,
domination, and independence constraints:
\begin{equation}\label{eq:sigma}
 \sigma(D):\quad
 \begin{array}{ll}
 x_v\ge0 & \forall v\in V,\\
 x(N_D^+[v])\ge1 & \forall v\in V,\\
 x(Q)\le1 & \forall Q\in\mathcal C(D).
 \end{array}
\end{equation}
Its feasible set is the \emph{fractional kernel polytope},
\[
 FK(D)=\{x\in\R^V:x\text{ satisfies }\sigma(D)\}.
\]
The \emph{kernel polytope} $K(D)$ is the convex hull of kernel incidence vectors.
The integral points of $FK(D)$ are
exactly these incidence vectors.
\begin{definition}\label{def:km}
The digraph $D$ is \emph{kernel ideal} if $FK(D')=K(D')$ for every induced
subdigraph $D'$ of $D$.
It is \emph{kernel Mengerian} if $\sigma(D')$ is TDI for every induced
subdigraph $D'$.
It is \emph{kernel perfect} if each induced subdigraph has a kernel.
\end{definition}
Thus every kernel Mengerian digraph is kernel ideal.
An orientation is \emph{induced-minimal nonideal} if it is not kernel ideal
but every proper induced subdigraph is kernel ideal.

For a set $R$ of integer row vectors, $\mathbb N R$ and $\Z R$ denote
its nonnegative integer span and integer span, respectively.
A row is \emph{tight} at a point if its inequality holds with equality;
it is \emph{globally tight} if equality holds throughout the polytope.
An integer lattice is \emph{saturated} if it contains every integer
vector in its real span.

\subsection{Line multigraphs and clique-acyclicity}
\label{sec:line-ca}

The line multigraph $L(H)$ has vertex set $E(H)$, with as many edges
between two vertices as their root edges have common endpoints.
An \emph{orientation} assigns one direction to each line edge.
Thus a digon can occur only between parallel root edges; repeated arcs
in one direction contribute only once to an out-neighborhood.
For such an orientation we identify $V(D)$ with $E(H)$.
This identification will be used for structural arguments; the kernel
system \eqref{eq:sigma} is defined entirely on $D$.
An induced subdigraph selects root edges and retains every arc between
them. The selected root need not be vertex-induced in $H$.

An arc $u\to v$ is a \emph{one-way arc} if $v\to u$ is absent.
The orientation is \emph{clique-acyclic} if the one-way arcs
within each clique form an acyclic digraph.
Every clique has a kernel: a sink of its one-way arcs is dominated
by every other vertex of the clique.
Root degrees and parallel multiplicities are unrestricted throughout.

A \emph{preference system} $(H,\prec)$ consists of $H$ and a strict linear
order $\prec_u$ on each $\delta_H(u)$, where $f\prec_u e$ means that $u$
\emph{prefers} $f$ to $e$.
It is a \emph{preference realization} of $D$ if orienting each line edge
at $u$ from $e$ to $f$ whenever $f\prec_u e$ gives $D$, including its
arc multiplicities.
The root-edge vertices remain fixed, but the common-endpoint labels of
parallel line edges may be exchanged.
Opposite endpoint orders on a parallel pair realize a digon.

Clique-acyclicity guarantees the fractional feasibility needed in the proof.

\begin{lemma}[Aharoni and Holzman {\cite[Theorem~1.1]{AharHolz98}}]\label{lem:fk-nonempty}
If $D$ is clique-acyclic, then $FK(D)$ is nonempty.
\end{lemma}
\begin{proof}
Apply the cited theorem to $D$ with all arcs reversed.
Their convention uses closed in-neighborhoods; reversal gives our closed
out-neighborhood convention and preserves clique-acyclicity.
Their strong fractional kernel in particular satisfies \eqref{eq:sigma}.
\end{proof}

To ensure an integral feasible point on every induced restriction,
one needs a further condition on directed cycles.
For line multigraphs, the following theorem characterizes this property.

\begin{theorem}[Borodin \textit{et al.} \cite{BoroKost98}]
\label{thm:bkw}
Let $D$ be a clique-acyclic orientation of $L(H)$.
Then $D$ is kernel perfect if and only if there is no odd cycle $C$ in
$H$ for which $D[E(C)]$ is a directed cycle.
\end{theorem}

In particular, every clique-acyclic orientation with a bipartite root is kernel perfect.

\subsection{Matching and independence constraints}
\label{sec:star-reduction}

For a clique-acyclic orientation, independence constraints can be recovered
from matching and domination constraints by integral substitutions.
We use these substitutions for the support and row-lattice arguments.
First recall the following form of the tight-row criterion for TDI;
see also \cite{Schr03}.

\begin{lemma}[Giles and Pulleyblank \cite{GilesPull79}]\label{lem:tdi-vertex}
Let $P=\{x:Ax\le b\}$ be a nonempty pointed polyhedron with $A$ integral.
For a vertex $x$, let $R_x$ be the rows of $A$ tight at $x$.
The system is TDI if and only if
\begin{equation}\label{eq:hilbert}
 \operatorname{cone}(R_x)\cap\Z^n=\mathbb N R_x
 \qquad\text{for every vertex }x.
\end{equation}
\end{lemma}
\begin{proof}
An objective belongs to $\operatorname{cone}(R_x)$ exactly when $x$ maximizes
it over $P$.
A nonnegative row representation supported on $R_x$ gives a dual vector
whose objective value equals its scalar product with $x$, hence is optimal.
Conversely complementary slackness places every optimal dual vector on
rows tight at $x$.
Every finite optimum over a nonempty pointed polyhedron is attained at a
vertex, so these local statements give the claimed equivalence.
\end{proof}

We next reduce the independence constraints to matching constraints.
Fix an orientation $D$ of $L(H)$, and let
$\mathcal S_H\subseteq\mathcal C(D)$ be the family of nonempty cliques
corresponding to full root stars under $V(D)=E(H)$.
The constraints $x(\delta_H(u))\le1$ are the matching constraints.
Retaining just these among the independence constraints gives the
\emph{star subsystem}
\begin{equation}\label{eq:direct}
 \pi^*(D,H):\quad
 \begin{array}{ll}
 x_v\ge0 & \forall v\in V(D),\\
 x(N_D^+[v])\ge1 & \forall v\in V(D),\\
 x(Q)\le1 & \forall Q\in\mathcal S_H.
 \end{array}
\end{equation}
This is a subsystem of the kernel system on $D$.
The next lemma gives the star description and the integer row identities
used in the proof.
With every row written in $\le$ form, an omitted row that is tight
will be a nonnegative integral combination of retained tight rows.
A clique in a line multigraph is contained either in a root star or in the
full edge set $O$ on three root vertices.
The set $O$ includes all parallel copies on those three root vertices.

\begin{lemma}[Redundancy with an integral tight span]\label{lem:redundant}
For clique-acyclic $D$, systems $\sigma(D)$ and $\pi^*(D,H)$ define the same
polytope, and every omitted independence row has an integral tight-span
representation in the star subsystem.
\end{lemma}
\begin{proof}
A sub-star row is its full star row plus nonnegativity rows for the omitted
edges.
If the sub-star is tight, the full star is tight and those omitted
coordinates vanish.

Consider the full triangle clique $O=E(H[\{u,v,w\}])$.
Choose a sink $r$ of the one-way arcs of $D[O]$, say $r=uw$.
Set $A=\delta_H(u)$, $B=\delta_H(w)$, and $N=N^+[r]$.
The vector
\[
 p=\chi_A+\chi_B-\chi_N-\chi_O
\]
has nonnegative integral coordinates.
For a member of $O$ not parallel to $r$, this follows because $r$ has no
outgoing one-way arc to that member, so the coordinate is $1-0-1=0$.
For a member parallel to $r$, it is $2-\mathbf1_N-1\ge0$.
For an edge outside $O$ incident to $u$ or $w$, it is
$1-\mathbf1_N\ge0$; elsewhere it is zero.
Thus
\begin{equation}\label{eq:triangle-span}
 \chi_O=\chi_A+\chi_B-\chi_N+\sum_e p_e(-\chi_e).
\end{equation}
The right-hand sides of the corresponding inequalities sum to $1+1-1=1$,
proving $x(O)\le1$.
If $x(O)=1$, equality forces both matching constraints and the domination at
$r$ to be tight, and forces $x_e=0$ whenever $p_e>0$.
Equation~\eqref{eq:triangle-span} is therefore an integral tight-span
representation.
Subsets of $O$ follow by adding nonnegativity rows.
\end{proof}

\section{Kernel ideal and kernel Mengerian orientations}
\label{sec:kernel}

We now characterize hereditary integrality and total dual integrality
of the linear system $\sigma(D)$.
Throughout this section, $D$ is a clique-acyclic orientation of $L(H)$
unless stated otherwise.

\subsection{Cycle packings and the characterization}
\label{sec:configurations}
\label{sec:general-characterization}
\label{sec:packing-bipartite}

We describe the obstruction through directed root cycles and edges
that connect their alternating choices.
The auxiliary graph below records both cycle adjacencies and these
connections.

A \emph{directed root-cycle packing} is a set $S\subseteq E(H)$ whose
nonisolated root components are vertex-disjoint cycles, each either a
parallel pair inducing a digon in $D$, or a cycle $C$ of length at
least four for which $D[E(C)]$ is a directed cycle.
The empty packing is allowed;
Clique-acyclicity excludes directed root triangles.
Its \emph{two-target connectors} form the set
\[
 T_D(S)=\{z\in E(H)\setminus S:|N_D^+(z)\cap S|=2\}.
\]
For $T\subseteq T_D(S)$, define an ordinary undirected graph $B_D(S,T)$
with vertex set $S$. Join consecutive edges of each support cycle,
and join the two members of $N_D^+(z)\cap S$ for each $z\in T$.
Repeated auxiliary edges are suppressed; a parallel support pair
contributes a single auxiliary edge.
Write $B_D(S)=B_D(S,T_D(S))$.
Whenever cycles and connectors are selected, we retain all induced
arcs in $D[S\cup T]$, including arcs between connectors.
The three graphs have different roles: packing edges and connectors
are edges of $H$ and vertices of $D$, whereas in $B_D(S,T)$ the
packing edges are vertices and each connector contributes an edge.

\begin{definition}[Even-cycle obstruction]\label{def:cycle-obstruction}
Let $S$ be a nonempty directed root-cycle packing consisting of even
cycles, and let $T\subseteq T_D(S)$.
The selected cycles and connectors form an \emph{even-cycle obstruction}
if every connected component of $B_D(S,T)$ contains an odd cycle.
The corresponding obstruction is the induced orientation $D[S\cup T]$.
\end{definition}

A single nonbipartite component suffices to select an obstruction:
retain the support cycles and connectors belonging to that component.

\begin{theorem}\label{thm:general-characterization}
\label{thm:packing-bipartite}
Let $D$ be a clique-acyclic orientation of $L(H)$, where $H$ is
a loopless multigraph. The following are equivalent:
\begin{enumerate}[label=(\roman*)]
\item $D$ is kernel ideal;
\item $D$ is kernel Mengerian;
\item $B_D(S)$ is bipartite for every directed root-cycle packing $S$;
\item $D$ is kernel perfect and has no even-cycle obstruction.
\end{enumerate}
Every nonbipartite component gives a nonintegral $1/2$-integral vertex
on the induced orientation consisting of its support cycles and
their two-target connectors.
\end{theorem}

Condition~(iv) separates hereditary kernel existence from integrality.
Each edge $ef$ of $B_D(S)$ requires opposite binary choices on $e$ and $f$.
An odd auxiliary cycle makes these choices inconsistent and forces the
fractional value $1/2$.
Thus even support cycles may yield an obstruction when their connector
requirements cannot be satisfied simultaneously.

\paragraph{Kernel perfectness does not imply integrality.}
The connector condition remains necessary on bipartite roots, where
Theorem~\ref{thm:bkw} already guarantees kernel perfectness.
The next example also shows that clique-acyclic orientations need not
arise from preference systems.

\begin{proposition}\label{thm:mixed-obstruction}
There is a kernel-perfect orientation on a bipartite root that is not
kernel ideal and admits no global preference realization.
\end{proposition}
\begin{proof}
Let
\[
 f_0=01, f_1=12, f_2=23, f_3=30,\qquad
 d=d'=45,\qquad w_0=05, w_1=14.
\]
Distinct names $d,d'$ denote the two parallel copies.
Orient all line-multigraph adjacencies by
\[
\begin{gathered}
 f_0\to f_1\to f_2\to f_3\to f_0,\qquad d\rightleftarrows d',\\
 f_3\to w_0\to f_0,\qquad f_0\to w_1\to f_1,\\
 d'\to w_0\to d,\qquad d'\to w_1\to d.
\end{gathered}
\]
The root bipartition is $\{0,2,4\}\cup\{1,3,5\}$.
The nontrivial maximal stars are
$\{f_0,f_3,w_0\}$, $\{f_0,f_1,w_1\}$,
$\{d,d',w_0\}$, and $\{d,d',w_1\}$, and their one-way arcs are
acyclic; the remaining maximal cliques are two-edge stars.
Thus $D$ is clique-acyclic.
Its bipartite root makes it kernel perfect by Theorem~\ref{thm:bkw}.

The fractional polytope can be computed explicitly. Domination at $d$
and the star at $4$ force $x_d+x_{d'}=1$ and $x_{w_1}=0$;
the star at $5$ then gives $x_{w_0}=0$.
The four support stars and dominations give alternating values on the
four-cycle. The remaining connector rows give
\[
\begin{split}
 FK(D)=\{&(t,1-t,t,1-t,s,1-s,0,0):\\
        &0\le t\le1,\quad \max\{t,1-t\}\le s\le1\},
\end{split}
\]
where the coordinates are ordered as $f_0,f_1,f_2,f_3,d,d',w_0,w_1$.
These conditions also imply every independence row: the bipartite root has
only star cliques. The point $(t,s)=(1/2,1/2)$ is a vertex, determined by the two tight
bounds $s=t$ and $s=1-t$.
It gives a fractional vertex of $FK(D)$, so $D$ is not kernel ideal.

At root vertex $5$, the comparisons $d'\to w_0\to d$ force $d$ to be
preferred to $d'$; the comparisons $d'\to w_1\to d$ impose the same
order at vertex $4$.
Such endpoint orders would orient both line edges from $d'$ to $d$,
contrary to the digon. Hence no preference realization exists.
\end{proof}

Here $S=\{f_0,f_1,f_2,f_3,d,d'\}$ is an even directed root-cycle packing.
The connectors $w_0$ and $w_1$ have targets $\{f_0,d\}$ and $\{f_1,d\}$,
respectively, so $B_D(S)$ contains the triangle $f_0,f_1,d,f_0$.
Thus the example is an even-cycle obstruction in
Theorem~\ref{thm:general-characterization}.

\section{Proof of the characterization}
\label{sec:kernel-proof}
\label{sec:support-reductions}

Throughout this section, $D$ is a clique-acyclic orientation of $L(H)$.
The proof of Theorem~\ref{thm:general-characterization} has two tasks:
extract a cycle obstruction from every failure of integrality, and
construct integer dual certificates whenever all induced restrictions
are integral.
Saturation recovers compatible endpoint orders when each coordinate
can be positive at a feasible point.
Globally zero coordinates retain their domination rows when eliminated.

The compatible orders and a one-row cut give $1/2$-integrality in
minimal failures, whose tight rows prove the equivalence of
(i), (iii), and (iv).
Section~\ref{sec:dual-proof} proves (i) $\Rightarrow$ (ii) through integer
row lattices and zero-coordinate lifting; the converse follows from TDI.

\subsection{Saturated stars}

With the star description available, we can compare total domination
load with total endpoint load by counting each root edge at its two
endpoints.

\begin{lemma}[Saturation]\label{lem:saturation}
Let $x$ be feasible for the star subsystem of $D$, and put
$s_v=x(\delta_H(v))$ and $q_e=x(N^+[e])$.
Then $s_v\in\{0,1\}$, and $x_e>0$ implies $q_e=1$.
Moreover, $x_ex_f=0$ whenever parallel root edges $e,f$ have all their
line arcs in the same direction.
\end{lemma}
\begin{proof}
Let $\mathcal P_1$ be the set of unordered parallel root-edge pairs whose
line arcs all have the same direction. We first show that
\begin{equation}\label{eq:saturation}
 \sum_e x_eq_e=\frac12\sum_v s_v^2-
       \sum_{\{e,f\}\in\mathcal P_1}x_ex_f.
\end{equation}
Expand the two quadratic sums.
A square $x_e^2$ has coefficient one in each.
A pair sharing exactly one root endpoint contributes $x_ex_f$ to each.
A parallel pair contributes $2x_ex_f$ to the half-sum of star squares,
and contributes $2x_ex_f$ to the neighborhood sum for a digon, but only
$x_ex_f$ when its arcs have one direction.
This proves \eqref{eq:saturation}.
Since $q_e\ge1$, $0\le s_v\le1$, and every root edge has two endpoints,
\[
 \sum_e x_e\le\sum_e x_eq_e
 \le\frac12\sum_v s_v^2
 \le\frac12\sum_v s_v=\sum_e x_e.
\]
Equality throughout makes each of
$\sum_e x_e(q_e-1)$, $\sum_{\mathcal P_1}x_ex_f$, and
$\sum_v s_v(1-s_v)$ vanish.
All their summands are nonnegative, giving the stated conclusions.
\end{proof}

\subsection{Preference orders on positive supports}
\label{sec:positive-supports}

When every coordinate is positive at some feasible point, averaging
gives one point that is positive everywhere.
Together with clique-acyclicity, the resulting tight equalities make the
star orders compatible across their common parallel classes, as follows.

\begin{lemma}[Compatible endpoint orders]\label{lem:active-orders}
Suppose $D$ is clique-acyclic, $FK(D)$ is nonempty, and each coordinate is
positive at some feasible point.  Then all parallel pairs are digons, and
there are strict linear orders on the root stars such that
\begin{enumerate}[label=(\alph*)]
\item for nonparallel incident edges $e,f$, the arc $e\to f$ means that
      $f$ precedes $e$ in their common star order;
\item the members of each parallel class occur in opposite orders at its
      two endpoints.
\end{enumerate}
Consequently, if $e=uv$ occupies positions $i,j$ in these orders, its
closed domination row is $x_e$ plus the two preceding star prefixes,
with no repeated coordinate.
\end{lemma}
\begin{proof}
Average finitely many feasible points to obtain $\bar x>0$.
Lemma~\ref{lem:saturation} gives the digons, unit loads at all nonisolated
root vertices, and tight domination at every edge.
For a parallel class $C$ on $uv$, put
\[
 O_u(e)=N_D^+(e)\cap(\delta_H(u)\setminus C)\qquad\forall e\in C,
\]
and define $O_v(e)$ similarly.
The sets $O_u(e)$ are nested: incomparable sets would give a directed
four-cycle of one-way arcs alternating between two members of $C$ and two
other edges in the star at $u$, contradicting clique-acyclicity.
The same holds at $v$.
Tight domination gives
\[
 \bar x(O_u(e))+\bar x(O_v(e))=1-\bar x(C)\qquad\forall e\in C.
\]
Strict positivity implies that strict inclusion at one endpoint is paired
with reverse strict inclusion at the other, and that equality of sets
at one endpoint is paired with equality at the other.

In a topological order of the one-way arcs of a root star, every
strict inclusion between these out-neighborhoods has a forced direction:
an edge in the difference gives a path of two one-way arcs.
Members of $C$ with identical external out-neighborhoods have identical
one-way in- and out-neighborhoods within the star; there are no
one-way arcs between them. Contract each such twin class.
A topological order of the quotient can be expanded by any ordering
inside a contracted class. Between two distinct classes belonging to
$C$, strict inclusion fixes the relative order, and the preceding
identity forces the opposite relative order at the other endpoint.
The partitions into twin classes are also the same at both endpoints.
Choose an arbitrary order inside each class at one endpoint and the
reverse order at the other. These choices are independent across
parallel classes, since changing the order inside a twin class changes
no one-way comparison with an outside edge.
Reversing the expanded topological orders gives the stated preference
convention and opposite orders on every parallel class.
Every nonparallel neighbor is counted at its unique common endpoint,
and every parallel alternative precedes $e$ at exactly one endpoint.
This proves the last assertion.
\end{proof}

\subsection{Zero coordinates and retained blocking constraints}
\label{sec:retained-support}

Eliminating a globally zero coordinate must retain every original row,
including the domination inequality centered there.
The following lemma justifies lifting integral dual solutions after
this elimination.

\begin{lemma}[Chen \textit{et al.} {\cite[Lemma~3.3]{ChenDing12}}]\label{lem:zero-lift}
Let $A,B$ be integral matrices, $b$ rational, and
\[
 P=\{(x,z):Ax+Bz\le b,\ x,z\ge0\}\ne\varnothing.
\]
Suppose $z=0$ throughout $P$.
If the system $Ax\le b$, $x\ge0$, retaining every original row, is TDI,
then the system defining $P$ is TDI.
\end{lemma}
\begin{proof}
Let $(w,\bar w)$ be integral with finite optimum.
Choose an optimal multiplier $y^*\in\Z_{\ge0}$ for the restricted
system, so $A^Ty^*\ge w$.
Put $c=\max\{\bar w-B^Ty^*,0\}$ coordinatewise.
The objective $c^Tz$ has optimum zero on $P$.
Rational LP duality supplies a rational $\lambda\ge0$ with
\[
 A^T\lambda\ge0,\qquad B^T\lambda\ge c,\qquad b^T\lambda=0.
\]
Multiply $\lambda$ by a positive integer to make it integral.
The inequalities remain valid because their right-hand sides are nonnegative.
Then $y=y^*+\lambda$ satisfies both column inequalities for $(w,\bar w)$
and has value $b^Ty^*$, the common optimum.
Its nonnegativity multipliers $A^Ty-w$ and $B^Ty-\bar w$ are nonnegative
and integral, since $A,B$ are integral.
Thus it gives the required full dual solution.
\end{proof}

\subsection{\texorpdfstring{$1/2$}{1/2}-integrality in minimal failures}
\label{sec:half-integrality}

The reconstructed orders turn star and domination equalities into
relations between pairs of prefix sums.
This serves both parts of the proof: it bounds vertex denominators
by two and, when a kernel exists, proves saturation of the global
row lattice needed for integer duals.

\begin{lemma}[Prefix equations]\label{lem:prefix-equations}
Under the hypotheses of Lemma~\ref{lem:active-orders}, every vertex of
$FK(D)$ is $1/2$-integral.
If $FK(D)$ contains a kernel, the integer lattice generated by all its
globally tight original rows is saturated in its real span.
\end{lemma}
\begin{proof}
Ignore isolated root vertices, whose star rows are harmless constants.
For each ordered star $e_{v,1},\ldots,e_{v,d(v)}$, introduce prefixes
\[
 t_{v,0}=0,\qquad t_{v,i}=\sum_{h=1}^i x_{e_{v,h}}.
\]
Average feasible points to obtain $\bar x>0$.
For any feasible $x$, the midpoint $y=(x+\bar x)/2$ is positive everywhere.
Lemma~\ref{lem:saturation} makes every nonisolated star and every
domination row tight at both $y$ and $\bar x$.
Since $x=2y-\bar x$, these rows are also tight at $x$.
Thus all stars and dominations are equalities throughout $FK(D)$.
By Lemma~\ref{lem:redundant}, the polytope is therefore exactly the
star and domination equalities together with $x\ge0$.
For $e=uv$ at positions $i,j$, the equations are equivalently
\begin{equation}\label{eq:prefix-pairs}
 t_{u,i}+t_{v,j-1}=1,\qquad
 t_{u,i-1}+t_{v,j}=1,
\end{equation}
with boundary values $t_{v,0}=0,t_{v,d(v)}=1$ and monotonicity
$t_{v,i-1}\le t_{v,i}$.
The difference of the two pair equations makes the two endpoint
increments equal; their common value recovers $x_e$.
Thus this is an affine bijection with integral formulas in both
directions, and it preserves vertices.

Every row of this prefix system has at most two nonzero coefficients,
each equal to $1$ or $-1$, and an integral right-hand side.
At a vertex, join variables that occur in a tight two-variable row.
Along each connected component the tight equations express every variable
as $\varepsilon t+a$, where $\varepsilon\in\{1,-1\}$ and $a\in\Z$.
A tight singleton row fixes an integral value; an inconsistent sign around
a cycle fixes $2t$ to an integer.
Without either event a component permits a nonzero two-sided perturbation,
contradicting vertexhood.
All prefixes, and hence all edge differences $x_e$, are $1/2$-integral.

For the lattice assertion, let $G_0$ consist of the star and domination
normal rows.  Adding the prefix definitions is an integer identity-block
extension of their row lattice: change coordinates from $t$ to
$t-\operatorname{prefix}(x)$.
This unimodular change shows that saturation is preserved.
There is also an integer row equivalence to \eqref{eq:prefix-pairs},
the boundary rows, and one increment definition per root edge.
For each edge use the three coefficient rows
\[
 U_e=x_e-t_{u,i}+t_{u,i-1},\quad
 V_e=x_e-t_{v,j}+t_{v,j-1},\quad
 D_e=x_e+t_{u,i-1}+t_{v,j-1}.
\]
The replacement by $U_e,D_e-U_e,D_e-V_e$ uses the integer matrix
\[
 \begin{pmatrix}1&0&0\\-1&0&1\\0&-1&1\end{pmatrix},
 \qquad \det=1.
\]
The initial boundary rows $t_{v,0}=0$ are retained throughout.
Star rows become terminal boundary rows $t_{v,d(v)}=1$ using the
increment definitions; all operations have integer inverses.
Repeated star rows may be retained, since they do not change the
generated lattice.
Eliminating $x_e$ using the retained unit pivot $U_e$ leaves precisely
the pair-sum rows and boundary unit rows.

A kernel gives integer prefixes, each zero or one.
It is therefore a bipartition of the graph whose edges are the pairs in
\eqref{eq:prefix-pairs}.
After column sign changes on one bipartition class, the pair rows have
one $+1$ and one $-1$.
Within a connected component their integer span consists exactly of the
integer vectors with coordinate sum zero: paths generate every difference
of two unit vectors.
If a boundary unit row occurs in that component, the integer span is the
whole coordinate lattice instead.
Both lattices are saturated, so the preceding unimodular operations prove
the same for $\Z G_0$.
Finally, a globally tight independence row has an integral tight-span
representation in star-system rows by Lemma~\ref{lem:redundant}.
At $\bar x>0$ no nonnegativity row is tight, so this representation uses
only $G_0$.
Thus adding all globally tight independence rows does not change this lattice.
\end{proof}

A minimal nonideal orientation may also have an identically zero
coordinate $x_z$.
Eliminating it retains its domination row on the integral polytope
$FK(D-z)$.
Since a kernel is a root matching, the retained row has values in
$\{0,1,2\}$ at the vertices of this smaller polytope.
The next lemma describes the resulting midpoint cut.

\begin{lemma}[One new domination]\label{lem:one-cut-half}
Let $Q$ be an integral polytope and $c$ an integral vector taking values
in $\{0,1,2\}$ at its vertices.
The polytope $Q\cap\{c^Tx\ge1\}$ is $1/2$-integral.
\end{lemma}
\begin{proof}
A vertex inherited from $Q$ is integral.
A new vertex lies on $c^Tx=1$ and in the relative interior of its minimal
face $F$ in $Q$.
If $\dim F\ge2$, a nonzero local direction within $F$ annihilates $c$,
contradicting vertexhood after the cut.
Thus $F$ is an edge.  An interior crossing of level one must have endpoint
values zero and two, so it occurs at their midpoint.
The new vertex is therefore $1/2$-integral as well.
\end{proof}

The positive-support argument and the single-row cut now cover both
cases of an induced-minimal failure.

\begin{theorem}[Minimal nonideal orientations]\label{thm:minimal-half}
Every clique-acyclic orientation that is induced-minimal among nonideal
orientations has a nonintegral $1/2$-integral vertex.
\end{theorem}
\begin{proof}
Let $D$ be such an orientation and put $P=FK(D)$.
The ambient polytope is nonintegral: otherwise its integrality and kernel
idealness of all proper induced subdigraphs would make $D$ kernel ideal.
Its nonemptiness also follows from Lemma~\ref{lem:fk-nonempty}.
If every coordinate can be positive, apply Lemma~\ref{lem:prefix-equations}.
Otherwise some coordinate $x_z$ vanishes throughout $P$.
Restrict it to zero, retaining every row.  The resulting polytope is
\[
 FK(D-z)\cap\{x(N_D^+(z))\ge1\}.
\]
All old domination rows restrict to those of $D-z$, and projected independence
rows are its independence rows, duplicates, or harmless constants.
The polytope $FK(D-z)$ is integral by minimality.
A kernel of $D-z$ meets $N_D^+(z)$ in at most two edges, because its root
edges form a matching and the root edge $z$ has two endpoints.
Lemma~\ref{lem:one-cut-half} gives $1/2$-integrality, preserved on restoring
the zero coordinate.
Since $P$ is nonintegral, the two cases provide the required nonintegral
$1/2$-integral vertex.
\end{proof}

\begin{lemma}[$1/2$-integral support]\label{lem:half-support}
Let $D$ be clique-acyclic and $x\in FK(D)\cap\{0,1/2,1\}^{E(H)}$.
Then $S=\{e:x_e=1/2\}$ is a directed root-cycle packing.
\end{lemma}
\begin{proof}
The matching constraints give degree at most two in $H[S]$.
An edge of value one cannot be adjacent to a member of $S$.
Domination therefore requires every member of $S$ to have an out-neighbor
in $S$.
A path component with $r$ edges has only $r-1$ available line adjacencies,
each oriented once, contradicting this requirement.
On a parallel pair both arcs are necessary.
A three-edge root cycle violates its independence constraint.
On any remaining root cycle there are exactly as many line adjacencies as
support edges, so every support edge has exactly one outgoing support arc:
the line cycle is directed.
\end{proof}

\subsection{Extracting the cycle obstruction}
\label{sec:primal-criterion}

Lemma~\ref{lem:half-support} supplies the directed root-cycle packing.
The choice constraints described after the kernel characterization now act on
perturbations: each relevant pair requires $d_e+d_f=0$.
An odd auxiliary cycle forces every perturbation in its component to
vanish; a bipartition supplies an alternating perturbation.
We use these two possibilities to prove the primal equivalences.

\begin{proof}[Proof of Theorem~\ref{thm:general-characterization}: the primal equivalences]
Suppose a component $W$ of $B_D(S)$ is nonbipartite.
Cycle adjacencies connect all edges of each support cycle, so $W$ consists
of the edge sets of complete support cycles.  Write $S_0=W$, and retain
the connectors $T_0\subseteq T_D(S)$ whose two targets belong to $S_0$.
No target pair straddles auxiliary components.
In $D[S_0\cup T_0]$, assign $1/2$ to $S_0$ and zero to $T_0$.
Every support domination contains its center and its unique support
successor; every connector domination contains two support targets.
No clique contains more than two support edges: three pairwise
intersecting packing edges would form a three-edge star or a root
triangle.  Thus the point is feasible.

The tight nonnegativity rows fix every connector coordinate at zero.
The support-star and connector rows then impose
$d_e+d_f=0$ for every edge $ef$ of $B_D(S)[W]$.
Sign propagation around an odd cycle forces all these perturbations to
vanish, since the component is connected.  Hence the point is a vertex
and witnesses failure of kernel idealness.  Connector-to-connector arcs
involve only fixed zero coordinates and have all been retained.

Conversely, suppose all the auxiliary graphs are bipartite but $D$ is
not kernel ideal.  Choose an induced-minimal nonideal restriction $J$ and a
nonintegral $1/2$-integral vertex $x$ of $FK(J)$, using
Theorem~\ref{thm:minimal-half}.
By Lemma~\ref{lem:half-support}, its $1/2$-integral support
$S=\{e:x_e=1/2\}$ is a nonempty directed root-cycle packing.

Choose bipartition signs $d_e\in\{1,-1\}$ on $S$ for $B_D(S)$, and set
$d=0$ on the other coordinates of $J$.
Every row of $\sigma(J)$ tight at $x$ annihilates $d$.
Indeed, a tight clique either contains a unit coordinate and no member
of $S$, or contains two consecutive edges of one support cycle.
A tight domination centered on $S$ contains its two consecutive support
edges.  Any other tight domination either contains a unit coordinate
and no member of $S$, or is centered at a zero coordinate and contains
exactly two targets in $S$.
Since $J$ is induced, that target pair is also an edge of $B_D(S)$.
Tight nonnegativity rows concern only zero coordinates.
Consequently $x\pm\epsilon d$ are feasible for sufficiently small
$\epsilon>0$, contradicting vertexhood.
This proves the equivalence of (i) and (iii).

Suppose (iii) holds.
Clique-acyclicity excludes a triangle $C$ in $H$ for which $D[E(C)]$ is a
directed cycle.
An odd cycle $C$ of length at least five for which $D[E(C)]$ is a directed
cycle would itself give a packing with a nonbipartite auxiliary graph.
Theorem~\ref{thm:bkw} therefore makes $D$ kernel perfect.
An even-cycle obstruction is also impossible: its auxiliary graph is
a subgraph of $B_D(S)$ and has a nonbipartite component.
Thus (iii) implies (iv).
Conversely, suppose (iv) holds and some $B_D(S)$ has a nonbipartite
component $W$.
Kernel perfectness forces all support cycles to be even by
Theorem~\ref{thm:bkw}.
The sets $S_0$ and $T_0$ defined above then form an even-cycle
obstruction, contradicting (iv). This proves (iv) implies (iii).
The first part of the proof also establishes the asserted
$1/2$-integral vertex certificate.
\end{proof}

\subsection{Retained domination rows and total dual integrality}
\label{sec:dual-proof}

The primal equivalences leave the implication from kernel ideal
to kernel Mengerian to be proved.
At a kernel, every original row has slack zero or one: domination has
one or two selected targets, while the independence and nonnegativity constraints
are immediate.
The following lemma reduces TDI to integer generation by the globally
tight original rows.

\begin{lemma}[Binary-slack systems]\label{lem:binary-slack}
Let $P=\{x:Ax\le b\}$ be a nonempty integral polytope, with $A,b$
integral.  Suppose each original row has slack zero or one at every vertex.
Let $G$ be the original normal rows tight everywhere on $P$.
Then this system is TDI if and only if
\[
 \Z G=\operatorname{span}_{\R}(G)\cap\Z^n.
\]
\end{lemma}
\begin{proof}
Put $L=\operatorname{span}(P-P)^\perp$.
At a relative interior point the normal cone is $L$ and its original tight
rows are exactly $G$, so $\operatorname{cone}(G)=L$.
Each $g\in G$ has a nonnegative rational representation
$-g=\sum_h\lambda_hh$ using $G$.
Clear denominators by an integer $m\ge1$ to obtain
\[
 -g=\sum_h m\lambda_hh+(m-1)g.
\]
Hence $\mathbb N G=\Z G$.

Fix a vertex $v$, write $R$ for its original tight rows, and
$C=\operatorname{cone}(R)$ for its normal cone.
For each vertex $u$ of $P$ set $\ell_u=v-u$.
Then $C=\{w:\ell_u^Tw\ge0\text{ for all }u\}$, and
$\ell_u^Tr\in\{0,1\}$ for every $r\in R$.
Take an integral $w\in C\setminus L$ and express it as a nonnegative
real combination of $R$.
Some row $r$ with positive coefficient has
$g(r)>0$, where $g(a)=\sum_u\ell_u^Ta$.
If $\ell_u^Tw=0$, that row satisfies $\ell_u^Tr=0$.
Otherwise integrality gives $\ell_u^Tw\ge1\ge\ell_u^Tr$.
Thus $w-r\in C\cap\Z^n$, and the nonnegative integer grade $g$ strictly
decreases.  Repetition terminates at $w_0\in L\cap\Z^n$.
If $\Z G$ is saturated, $w_0\in\mathbb N G$, so $w\in\mathbb N R$.
This proves the Hilbert property of every vertex's original tight rows,
equivalently TDI by Lemma~\ref{lem:tdi-vertex}.

Conversely, a lattice point of $L$ missing from $\Z G$ cannot be generated
integrally by the tight rows at any vertex.
Indeed, pair such a representation with a relative interior point minus
that vertex: the contribution of every nonglobal tight row is strictly
negative, so its coefficient must be zero.
This contradicts the Hilbert property required by TDI.
\end{proof}

Lemma~\ref{lem:prefix-equations} gives the required lattice saturation
when each coordinate is positive somewhere.
The next step treats a globally zero coordinate with its domination
row retained.

\begin{lemma}[A zero-coordinate step for TDI]\label{lem:zero-step-tdi}
Suppose $FK(D)$ is nonempty and integral, $x_z=0$ throughout it, and
$\sigma(D-z)$ is TDI.  Then $\sigma(D)$ is TDI.
\end{lemma}
\begin{proof}
Write $Q=FK(D-z)$ and $c=\chi_{N_D^+(z)}$.
After retaining all rows and deleting column $z$, the system is that of
$Q$ together with $c^Tx\ge1$, up to duplicate independence rows and constants.
Its polytope $P$ is nonempty and integral, and the TDI hypothesis makes
$Q$ integral. Nonemptiness of $P$ gives $\max_Q c^Tx\ge1$.
If $\max_Qc^Tx>1$, mix a point of $Q$ satisfying $c^Tx>1$ with a
point in the relative interior of $Q$, giving the latter a sufficiently
small positive weight.
The result is a relative-interior point still satisfying $c^Tx>1$.
A small relative neighborhood of this point lies in $P$, so $P$ and
$Q$ have the same affine hull.
The new inequality is not globally tight, and no previously nonglobal
row becomes globally tight.
The globally tight row lattice is therefore unchanged from that of $Q$.
It is saturated by the necessary direction of
Lemma~\ref{lem:binary-slack}, since the integral $FK(D-z)$ has binary slacks.
All rows have binary slacks at the kernels of $P$, so
Lemma~\ref{lem:binary-slack} proves TDI.

Otherwise $\max_Qc^Tx=1$, and $P$ is the exposed face $c^Tx=1$.
For an integral objective $w$, choose an integer $M\ge0$ sufficiently
large that every maximizer of $w+Mc$ over $Q$ lies in $P$.
Such an $M$ exists because $Q$ has finitely many vertices.
An integral optimal dual for this objective, together with multiplier
$M$ on the new row $-c^Tx\le-1$, is an integral optimal dual for $w$ on
$P$: its value is the optimum on $Q$ minus $M$.
Finally Lemma~\ref{lem:zero-lift} restores column $z$.
\end{proof}

\begin{samepage}
We can now complete the kernel characterization by induction on the number of
root edges.

\begin{proof}[Proof of Theorem~\ref{thm:general-characterization}: the dual equivalence]
TDI with integral right-hand sides implies integrality, so (ii) implies (i).
To prove the reverse implication, induct on $|V(D)|$ for kernel ideal $D$.
The empty digraph gives the trivial zero-dimensional system.
Every proper induced subdigraph is kernel ideal and hence kernel Mengerian
by induction. Lemma~\ref{lem:fk-nonempty} makes $FK(D)$ nonempty.
If some coordinate is identically zero, apply
Lemma~\ref{lem:zero-step-tdi} and induction on $D-z$.
Otherwise Lemma~\ref{lem:prefix-equations} makes the globally tight row
lattice saturated: integrality and nonemptiness supply a kernel.
The binary-slack lemma proves TDI of $\sigma(D)$.
Together with the induction hypothesis this makes $D$ kernel Mengerian.
\end{proof}
\end{samepage}

\section*{Acknowledgments}
We are grateful to Prof. Wenan Zang for valuable suggestions.

\end{document}